\documentclass[letterpaper, 10 pt, conference]{IEEEtran}
\usepackage{graphicx}
\usepackage{tikz,pmat}
\usepackage{amsmath} 
\usepackage{amssymb}  
\usepackage{amsfonts}
\usepackage{algorithm}
\usepackage{algpseudocode}
\usepackage{amsthm}
\usepackage{geometry}
\usepackage{subfig}
\newtheorem{theorem}{Theorem}[section]
\newtheorem{lemma}[theorem]{Lemma}
\newtheorem{proposition}[theorem]{Proposition}
\newtheorem{remark}{Remark}

\newtheorem{assumption}{Assumption}

\newcommand{\vertiii}[1]{{\left\vert\kern-0.25ex\left\vert\kern-0.25ex\left\vert #1 
		\right\vert\kern-0.25ex\right\vert\kern-0.25ex\right\vert}}

\title{\LARGE \bf
	Maximal Invariant Set Computation and Design for Markov Chains
}

\author{Dylan Janak and Beh\c{c}et A\c{c}\i{}kme\c{s}e
	\thanks{This work was supported by ...}
	\thanks{a1 is with the Department of A\&A,
		University of Washington, WA 98102, USA
		{\tt\small a1@gmail.com}}%
	\thanks{a2 is with the Department of A\&A,
		University of Washington, WA 98102, USA
		{\tt\small a2@gmail.com}}%
	\thanks{a3 is with Faculty of A\&A ,
		University of Washington, WA 98102, USA
		{\tt\small a3@gmail.com}}%
}
\title{Maximal Invariant Set Computation and Design for Markov Chains}
\author{Dylan Janak and Beh\c{c}et A\c{c}\i{}kme\c{s}e}

\begin{document}
	
	\maketitle
	
	\begin{abstract}
		
		We describe an algorithm for computing the maximal invariant set for a Markov chain with linear safety constraints on the distribution over states. We then propose a Markov chain synthesis method that guarantees finite determination of the maximal invariant set. Although this problem is bilinear in the general case, we are able to optimize the convergence rate to a desirable steady-state distribution over reversible Markov chains by solving a Semidefinite Program (SDP), which promotes efficient computation of the maximal invariant set. We then demonstrate this approach with a decentralized swarm guidance application subject to density upper bounds.
		
	\end{abstract}
	
	\section{Introduction}
	
	\subsection{Background}
	
	Set theory plays an important role in robust control design \cite{blanchini2008set}, and the concept of invariance is critical to ensure that safety constraints are verifiably satisfied. For a given dynamical system, a set of states is positively invariant if once the system enters that set, it will never leave. Safety requirements therefore can be proven by showing that the initial condition is in a positively invariant subset of the safe region. The problem of verifying that a particular initial condition will satisfy safety constraints for all future times has been studied for linear \cite{gilbert1991linear}\cite{kerrigan2001robust}\cite{rakovic2005} and polynomial \cite{hirata2008exact}\cite{magron2017semidefinite} systems by constructing either a positively invariant set or an explicit reachable or controllable set. It is often useful to characterize the set of all such initial conditions\textemdash also called the \textit{maximal output admissible set}\textemdash as the union of all positively invariant subsets of the safe region \cite{kolmanovsky1998theory}.
	
	We will consider discrete-time, time-invariant Markov chains with finitely many states. The particular sequence of states typically cannot be determined in advance because of the stochastic dynamics, however the probability vector over the set of states evolves as a deterministic, linear system. There has been much research to find optimal policies for Markov decision process \cite{puterman2014markov}\cite{arapostathis2003controlled}\cite{elchamie2016convex} which result in a closed-loop Markov chain. Semidefinite Programming (SDP) can aid in the design of Markov chains, specifically to constrain or optimize the mixing rate \cite{boyd1994linear}\cite{acikmese2015markov}\cite{acikmese2015TAC}\cite{boyd2009fastest}.
	

	In this paper, we combine results from polytopic invariance analysis \cite{blanchini1999survey} and Markov chain synthesis with safety constraints \cite{acikmese2015TAC}. We first specialize a set theoretic method of computing the maximal positively invariant set within a prescribed polytope to ergodic Markov chains. Under certain assumptions, the proposed algorithm requires a finite number of iterations to converge to the exact solution. We then propose an SDP-based synthesis method to compute a reversible Markov chain with sufficiently fast (or even optimal) mixing rate, subject to transition constraints, which ensures that the maximal admissible set can be exactly computed in finite time.

	\subsection{Notation}
	The notation $x[k]$ is used for a time-dependent vector $x$ evaluated at time $k \in \{0, 1, \dots\}$. When there is no explicit dependence on the time index, $x$ and $x^+$ may be used in place of $x[k]$ and $x[k+1]$, respectively. $\mathbf{1}$ is a column vector of ones, and $e_i$ is the $i^{\mathrm{th}}$ standard basis vector. $0$ represents either the scalar $0$, or a vector/matrix of all zeros. The inequality symbols $\geq$, $\leq$, $>$, and $<$ are interpreted elementwise for all entries of vectors and matrices, with $P\succ0$ ($P \succeq 0$) indicating that matrix $P$ is positive-definite (positive-semidefinite). A \textit{probability vector} $z$ satisfies $z \geq 0$ and $\mathbf{1}^T z = 1$. The set of probability vectors of a particular dimension is called the \textit{probability simplex}, and is written as $\Delta$. A \textit{Markov matrix}, $M$, is a nonnegative, square matrix such that $\mathbf{1}^T M = \mathbf{1}^T$. A \textit{polyhedron} is the intersection of finitely many half-spaces in $\mathbb{R}^n$, and is represented as $\mathcal{P}(G,g) := \{z \in \mathbb{R}^n \mid Gz \leq g\}$. The maximal positively invariant subset of $\mathcal{X}$ subject to $x^+ = f(x)$ is written as $\mathcal{O}_\infty(f,\mathcal{X})$, or $\mathcal{O}_\infty(M,\mathcal{X})$ for the linear system $x^+ = Mx$, as an abuse of notation. The binary operation $\odot$ denotes elementwise product of vectors or matrices, i.e., $(A\odot B)_{ij} = A_{ij}B_{ij}$. The \textit{spectral radius} of matrix $A$, denoted $\rho(A)$, is the maximum magnitude of its eigenvalues.
	
	\section{Problem Formulation}
	
	
	We consider a Markov chain over a finite set of states $\mathcal{S}=\{S_1,\dots,S_n\}$, and transition probabilities $M_{ij} = \mathrm{Pr}(s^+ = S_i \mid s = S_j)$. This process results in the deterministic LTI system $x^+ = Mx$ (i.e., $x[k+1]=Mx[k], \ \forall k\in \{0,1,2,\dots\}$), where $x_i[k]$ is the probability of being in state $S_i$ at time $k$, and $M$ is a Markov matrix. The transition matrix $M$ and the initial distribution $x[0]$ are treated as known, fixed quantities, so that the probability of encountering state $S_i$ at time $k$ is simply $e_i^T M^k x[0]$.
	
	Safety constraints take the form $Gx[k] \leq g$ for all $k \in \mathbb{N}$. For example, we could specify upper and lower bounds on the probability mass of each individual state, on the sum of probabilities over a subset of states, or on the difference between probability masses of adjacent states. Having $Gx[0] \leq g$ is not sufficient to ensure that these constraints are always satisfied, since a distribution may eventually violate this constraint. Therefore it is useful to characterize the set of all initial conditions which ensure safety for all subsequent times, simplifying the analysis from checking infinitely many constraints $Gx[0]\leq g, \ Gx[1]\leq g, \dots$, to checking if $x[0]$ satisfies a finite number of inequalities. 
	
	Given a dynamical system $x^+=f(x)$, a set $\mathcal{Y}$ is \textit{positively invariant} if $x \in \mathcal{Y} \Rightarrow f(x) \in \mathcal{Y}$. For any $\mathcal{X} \subset \mathbb{R}^n$, it can be shown that there exists a unique maximal positively invariant subset of $\mathcal{X}$ (sometimes simply called the \textit{maximal invariant set}) $\mathcal{O}_\infty(f, \mathcal{X})$ containing all positively invariant subsets of $\mathcal{X}$ \cite{tarski1955}. For a time-invariant Markov chain $x^+ = M x$, the problem of reducing safety constraints $Gx[k]\leq g, \ \forall k \geq 0$ to a set of conditions that depends only on $x[0]$ is equivalent to computing the maximal positively invariant subset of $\Delta\cap\mathcal{P}(G,g)$ with respect to the linear mapping $x \rightarrow Mx$. A related problem is how to design a Markov matrix on a given connected graph which ensures safety for all $x[0] \in \mathcal{X}_0$. This problem is more difficult to address because of the more complex relationship between the Markov matrix and the maximal invariant set, but a heuristic based on optimizing convergence rate may be obtained via semidefinite programming, as we will show in Section \ref{MarkovSynthesis}.

	\section{Invariant Set Computation}
	
	 For any asymptotically stable LTI system, there exists a positively invariant polyhedron which is bounded and nonempty \cite{bitsoris1988positively}\cite{Hennet1995}. It can also be shown that the maximal invariant subset of a polyhedron is also polyhedral, and can be computed exactly in finite time \cite{gilbert1991linear}\cite{farina1998invariant}. However, Markov chains are not asymptotically stable in the usual sense of the state approaching the origin. We first present the conceptual algorithm for computing the maximal invariant set for a general dynamical system, and then for the special case of a Markov chain with polyhedral safety constraints.
	 
	 \subsection{Conceptual Algorithm}
	 
	 The problem of finding the maximal invariant set for a Markov chain is a special case of finding the maximal invariant subset of $\mathcal{X} \subset \mathbb{R}^n$ for the dynamical system
	 \begin{equation} \label{eq:nonlinear}
	 x^+ = f(x).
	 \end{equation}
	 In fact, showing that $x[0]$ is in any positively invariant subset of $\mathcal{X}$ proves safety, because all subsequent states will remain in this set. It is better still to efficiently characterize all points which lie in one of these sets, i.e., to find the \textit{maximal} positively invariant subset of $\mathcal{X}$. The following lemma shows that the maximal invariant set is precisely the set of all initial conditions which remain in $\mathcal{X}$ for all time.
	 \begin{lemma}
	 	Given the system $x^+ = f(x)$ and set $\mathcal{X} \subset \mathbb{R}^n$, the safety condition $x[k] \in \mathcal{X} \ \forall k = 0,1,2,\dots$ is satisfied iff $x[0] \in \mathcal{O}_\infty(f,\mathcal{X})$, the maximal positively invariant subset of $\mathcal{X}$.
	 \end{lemma}
 	 A detailed proof of this lemma is omitted for space, but the idea is that $x[0] \in \mathcal{O}_\infty(f,\mathcal{X})$ implies that $\{x[0],x[1],\dots\}$ is a positively invariant subset of $\mathcal{X}$, hence $\{x[0],x[1],\dots\}$ is a subset of the maximal invariant set. Otherwise, $x[0]$ cannot be in an invariant subset of $\mathcal{X}$ because $x[k] \notin \mathcal{X}$ for some $k \geq 0$.
%
%
%

	 Since $\mathcal{O}_\infty(f,\mathcal{X})$ is the set of all initial conditions whose trajectories remain in $\mathcal{X}$ for all subsequent time steps, this set can be interpreted as the limit of the sequence $\mathcal{O}_{0}(f,\mathcal{X}), \ \mathcal{O}_{1}(f,\mathcal{X}), \ \dots$, where
	 \begin{equation}
	 \mathcal{O}_t(f,\mathcal{X}) = \{\xi \in \mathbb{R}^n \mid f^k(\xi) \in \mathcal{X}, \ k=0,\dots,t\}.
	 \end{equation}
	 Defining the preimage of set $\mathcal{A}$ for system \eqref{eq:nonlinear} as $\mathrm{Pre}(\mathcal{A})=\left\{\xi \in \mathbb{R}^n \mid f(\xi) \in \mathcal{A}\right\}$, we can express $\mathcal{O}_t$ as
	 \begin{equation}
	 \mathcal{O}_t(f,\mathcal{X}) = \mathcal{X} \cap \mathrm{Pre}(\mathcal{X}) \cap \dots \cap \mathrm{Pre}^{t}(\mathcal{X}).
	 \end{equation}
	 $\mathcal{O}_\infty(f,\mathcal{X})$ is said to be \textit{finitely determined} if there exists some $t^* \in \mathbb{N}$ such that $\mathcal{O}_{t^*}(f,\mathcal{X}) = \mathcal{O}_{\infty}(f,\mathcal{X})$. The following result and algorithm are slightly modified from those presented by Gilbert and Tan \cite{gilbert1991linear}.
	 \begin{lemma} \label{tstar}
	 	If there is some $t^*$ such that $\mathcal{O}_{t^*}(f,\mathcal{X}) \subseteq \mathcal{O}_{t^* + 1}(f,\mathcal{X})$, then $\mathcal{O}_\infty(f,\mathcal{X}) = \mathcal{O}_{t^*}(f,\mathcal{X})$.
	 \end{lemma}
	 \begin{algorithm}
	 	\caption{Maximal invariant subset of $\mathcal{X}$ for the system $x^+ = f(x)$}
	 	\label{generalalgorithm}
	 	\begin{algorithmic}
	 		\State $t \gets 0$
	 		\State $\mathcal{O}_0 \gets \mathcal{X}$
	 		\While {$\mathcal{O}_t \nsubseteq \mathrm{Pre}^{t+1}(\mathcal{X})$}
	 		\State $\mathcal{O}_{t+1} \gets \mathcal{O}_{t}\cap\mathrm{Pre}^{t+1}(\mathcal{X})$
	 		\State $t \gets t+1$
	 		\EndWhile
	 		\State $t^* \gets t$
	 		\State $\mathcal{O}_\infty(f,\mathcal{X}) \gets \mathcal{O}_{t^*}$
	 	\end{algorithmic}
	\end{algorithm}
 	Algorithm \ref{generalalgorithm} uses the property in Lemma \ref{tstar} to compute $\mathcal{O}_\infty(f,\mathcal{X})$. If this algorithm terminates, then the resulting $\mathcal{O}_{t^*}$ is the maximal invariant set. Therefore to determine if $x$ is in $\mathcal{O}_\infty$, one only needs to check the finite set of conditions $x \in \mathcal{X}, f(x) \in \mathcal{X}, \dots, f^{t^*}(x) \in \mathcal{X}$. However, the condition $\mathcal{O}_t \subseteq \mathrm{Pre}(\mathcal{O}_t)$ may be difficult to check for arbitrary $f$ and $\mathcal{X}$. Furthermore, this algorithm may never terminate, in which case the While loop may be stopped for some finite $t$ to give an outer approximation for $\mathcal{O}_\infty$ guaranteeing constraint satisfaction only for the next $t$ steps.

	 \subsection{Algorithm for Markov Chains with Polyhedral Constraints}
	 
	 Algorithm \ref{generalalgorithm} converges for particular sets of constraints and dynamics, but there is no general guarantee that $\mathcal{O}_\infty$ is finitely determined, even for linear systems \cite{gilbert1991linear}. Fortunately, finite determination is guaranteed for ergodic Markov chains with polyhedral safety constraints. From this section onward, the set of constraints is assumed to take the form $x[k] \in \mathcal{P}(G,g)$, and $x[0]$ is required to be a probability vector. We consider the system
	 \begin{subequations} \label{eq:linear}
	 	\begin{equation}
	 		x^+ = Mx,
	 	\end{equation}
	 	\begin{equation}
		 	\mathcal{X} = \Delta \cap \mathcal{P}(G,g),
	 	\end{equation}
	 \end{subequations}
	 where $M$ is a Markov matrix.
	 
	 To implement Algorithm \ref{generalalgorithm}, we first should be able to compute $\mathrm{Pre}^{t+1}(\mathcal{X})$, to determine if $\mathcal{O}_t$ is a subset of $\mathrm{Pre}^{t+1}(\mathcal{X})$, and to represent $\mathcal{O}_{t+1}$ as the set intersection $\mathcal{O}_t \cap \mathrm{Pre}^{t+1}(\mathcal{X})$. For System \eqref{eq:linear}, it is straightforward to show that $\mathrm{Pre}(\mathcal{P}(H,h))=\mathcal{P}(HA,h)$, and therefore that
	 \begin{equation}
	 \mathcal{O}_{t} =\mathcal{P}\left(\left[\begin{array}{c}G\\
	 GM\\ \vdots \\ GM^t\end{array}\right],\left[\begin{array}{c}g\\
	 g \\ \vdots \\ g\end{array}\right]\right).
	 \end{equation}
	 Since $\mathcal{O}_{t+1} = \mathcal{O}_t \cap \mathrm{Pre}(\mathcal{O}_t)$, it follows by induction that if $\mathcal{O}_t$ is a polyhedron, then so is $\mathcal{O}_k$ for all $k \geq t+1$. This also demonstrates that if the constraints $\mathcal{X}$ are polyhedral, and if $\mathcal{O}_\infty$ is finitely determined, then $\mathcal{O}_\infty = \mathcal{O}_{t^*}$ is a polyhedron.
	 
	 The following lemma demonstrates that checking the condition $\mathcal{O}_t \subseteq \mathrm{Pre}^{t+1}(\mathcal{X})$ in Algorithm \ref{generalalgorithm} is equivalent to solving a linear feasibility problem \cite{hennet1989extension}\cite{dorea1996computation}.
	\begin{lemma}[Extended Farkas' Lemma] \label{Farkas}
		Let $\mathcal{P}_1=\{z \in \mathbb{R}^n \mid G_1 z \leq g_1\}$ and $\mathcal{P}_2=\{z \in \mathbb{R}^n \mid G_2 z \leq g_2\}$, and let $\mathcal{P}_1$ be bounded and nonempty. Then $\mathcal{P}_1 \subseteq \mathcal{P}_2$ iff there exists a nonnegative matrix $Y$ such that $Yg_1 \leq g_2$ and $YG_1 = G_2$.
	\end{lemma}
	\begin{proof}
		(Sufficiency) Consider an arbitrary $x \in \mathcal{P}_1$, so by definition, $G_1 x \leq g_1$. Both sides of this inequality can be left-multiplied by the nonnegative matrix $Y$ to obtain $YG_1 x\leq Yg_1$, and thus $G_2 x = YG_1 x \leq Yg_1 \leq g_2$. This shows that $x \in \mathcal{P}_1 \Rightarrow x \in \mathcal{P}_2$, and so $\mathcal{P}_1 \subseteq \mathcal{P}_2$.
		
		(Necessity) These conditions are necessary as a direct result of strong duality. If $\mathcal{P}_1 \subseteq \mathcal{P}_2$, then $p_i^* = \sup_{x \in \mathcal{P}_1} e_i^T G_2 x \leq \sup_{x \in \mathcal{P}_2} e_i^T G_2 x \leq e_i^T g_2$.
		Slater's condition ensures that there is no duality gap, ensuring that the dual solutions $d_i^*$ are finite, and thus that the dual LPs are feasible. Since $-\infty < d_i^* = p_i^* \leq e_i^T g_2$, there must exist a minimizer $y_i^*$ in the dual domain for which $g_1^T y_i^* \leq e_i^Tg_2$. It then follows that the matrix $Y = \left[\begin{array}{ccc}	y_1^* & \cdots & y_{m_2}^*\end{array}\right]^T$ satisfies the conditions $Yg_1 \leq g_2$, $YG_1 = G_2$, and $Y \geq 0$.
	\end{proof}
	
	Lemma \ref{Farkas} is useful to verify that a polyhedron is positively invariant for linear dynamics, since the invariance condition  $Hx\leq h \Rightarrow HAx \leq h$ is equivalent to $\mathcal{P}(H,h) \subseteq \mathcal{P}(HA,h)$. The additional structure of the Markov chain dynamics on the probability simplex can be exploited to arrive at the following results.
	
	\begin{theorem} \label{farkasmarkov}
		Let $\Delta\cap\mathcal{P}(G,g)$ be nonempty. Then $\Delta\cap\mathcal{P}(G,g) \subseteq \mathcal{P}(H,h)$ iff there exists a nonnegative matrix $Y$ such that $H-h\mathbf{1}^T \leq Y(G-g\mathbf{1}^T)$.
	\end{theorem}
	\begin{proof}
		(Sufficiency) Let $x$ be any probability vector such that $Gx \leq g$. Since $x \geq 0$, the inequality $H-h\mathbf{1}^T \leq Y(G-g\mathbf{1}^T)$ can be right-multiplied by $x$ so that $(H-h\mathbf{1}^T)x \leq Y(G-g\mathbf{1}^T)x$, which simplifies to $Hx-h \leq Y(Gx-g)$. The right-hand side is the product of a nonnegative matrix $Y$ and a nonpositive vector $Gx-g$, so $Hx-h \leq Y(Gx-g) \leq 0$, which implies that $x \in \mathcal{P}(H,h)$.
		
		(Necessity) The condition $\Delta\cap\mathcal{P}(G,g) \subseteq \mathcal{P}(H,h)$ is equivalent to
		\begin{equation}
		\mathcal{P}\left(\left[\begin{array}{c}
		G \\ \mathbf{1}^T \\ -\mathbf{1}^T \\ -I
		\end{array}\right],\left[\begin{array}{c}
		g \\ 1 \\ -1 \\ 0
		\end{array}\right]\right) \subseteq \mathcal{P}(H,h).
		\end{equation}
		From Lemma \ref{Farkas}, if $\Delta\cap\mathcal{P}(G,g) \subseteq \mathcal{P}(H,h)$, then there exists a nonnegative matrix $\left[\begin{array}{cccc}Y&Y_+&Y_-&Y_0\end{array}\right]$ such that $YG + Y_+\mathbf{1}^T - Y_- \mathbf{1}^T - Y_0 = H$ and $Yg + Y_+ - Y_- \leq h$. The second inequality is equivalent to $(Y_+-Y_-)\mathbf{1}^T \leq (h-Yg)\mathbf{1}^T$. Putting these two inequalities together and noting that $Y_0 \geq 0$, we can verify that $YG+(h-Yg)\mathbf{1}^T \geq YG + Y_+\mathbf{1}^T - Y_- \mathbf{1}^T = Y_0 + H \geq H$, which implies that $Y(G-g\mathbf{1}^T) \geq H-h\mathbf{1}^T$.
	\end{proof}
	The next result gives a necessary and sufficient condition for invariance.
	\begin{theorem} \label{invariancetheorem}
		Consider the system $x^+ = Mx$ with $\mathbf{1}^T M = \mathbf{1}^T$, $M \geq 0$, and $x[0] \in \Delta\cap\mathcal{P}(G,g)$. The polyhedron $\Delta\cap\mathcal{P}(G,g)$ is positively invariant iff either $\Delta\cap\mathcal{P}(G,g) = \emptyset$, or there exists a nonnegative matrix $Y \geq 0$ such that $Y(G-g\mathbf{1}^T) \geq (G-g\mathbf{1}^T)M$.
	\end{theorem}
	\begin{proof}
		If $\Delta\cap\mathcal{P}(G,g) = \emptyset$, then no feasible trajectory exists, and the statement $x \in \emptyset \Rightarrow x^+ \in \emptyset$ is true by the nonexistence of a counterexample.
		
		If $\Delta\cap\mathcal{P}(G,g)$ is nonempty, then this set is positively invariant if $x \in \Delta\cap\mathcal{P}(G,g) \Rightarrow x^+ \in \Delta\cap\mathcal{P}(G,g)$. $M$ is a Markov matrix, so $Mx$ is a probability vector for all $x \in \Delta$. Therefore all that remains to be shown is the condition $x \in \Delta\cap\mathcal{P}(G,g) \Rightarrow Mx \in \mathcal{P}(G,g)$, or equivalently $\Delta\cap\mathcal{P}(G,g) \subseteq \mathcal{P}(GM,g)$. Theorem \ref{farkasmarkov} is then applied by setting $H=GM$ and $h=g$ to obtain the necessary and sufficient condition that there exists a nonnegative matrix $Y$ such that $Y(G-g\mathbf{1}^T) \geq GM-g\mathbf{1}^T$, or equivalently, $Y(G-g\mathbf{1}^T) \geq (G-g\mathbf{1}^T)M$.
		
	\end{proof}
	\begin{remark}
		For all $x \in \Delta$, $Gx\leq g$ is equivalent to the conical condition $(G-g\mathbf{1}^T)x \leq 0$, therefore $g$ can be set to zero without loss of generality.
	\end{remark}
	
	By Theorem \ref{farkasmarkov}, the condition $\mathcal{O}_t \subseteq \mathrm{Pre}^{t+1}(\mathcal{X})$ is equivalent to the existence of nonnegative matrices $Y_0,\dots,Y_t$ which satisfy the elementwise inequality
	\begin{equation}
	 Y_0(G-g\mathbf{1}^T)M^0+\dots+Y_t(G-g\mathbf{1}^T)M^t \geq (G-g\mathbf{1}^T)M^{t+1},
	\end{equation}
	which leads to Algorithm \ref{markovalgorithm} for computing the maximal invariant set for a Markov chain under polyhedral safety constraints.
	\begin{algorithm}
		\caption{Maximal invariant subset of $\mathcal{P}(G,g)$ for the Markov chain $x^+ = M x$}
		\label{markovalgorithm}
		\begin{algorithmic}
			\State $t \gets 0$
			\If {$\nexists x \geq 0 \ s.t. \ Gx \leq g, \mathbf{1}^Tx = 1 $}
				\State $\mathcal{O}_\infty = \emptyset$
			\Else
				\State $G_0 \gets G$
				\State $g_0 \gets g$
				\While {$\nexists Y \geq 0 \ s.t. \ Y(G_t-g_t\mathbf{1}^T) \geq GM^{t+1}-g\mathbf{1}^T$}
					\State $t \gets t+1$
					\State $G_t \gets \left[\begin{array}{c}
					G_{t-1} \\
					GM^t
					\end{array}\right]$
					\State $g_t \gets \left[\begin{array}{c}
					g_{t-1} \\
					g
					\end{array}\right]$
				\EndWhile
				\State $t^* \gets t$
				\State $\mathcal{O}_\infty \gets \mathcal{P}(G_{t^*},g_{t^*})\cap\Delta$
			\EndIf
		\end{algorithmic}
	\end{algorithm}
 	Figure \ref{fig:predemo} illustrates Algorithm \ref{markovalgorithm} for the simple 3-dimensional example with $x[k] \leq \left[\begin{array}{ccc}0.6 & 0.5 & 0.5\end{array}\right]^T$ and dynamics
	\begin{equation}
	x^+ = \left[\begin{array}{ccc}
	0.8 & 0.2 & 0 \\
	0.2 & 0.2 & 0.9 \\
	0 & 0.6 & 0.1
	\end{array}\right] x.
	\end{equation}
	
	\begin{figure}
		\centering
		\includegraphics[width=0.6\linewidth]{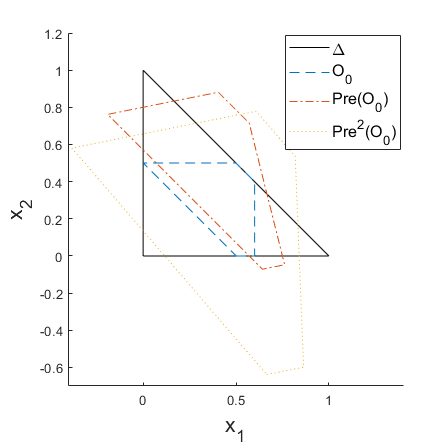}
		\caption{Algorithm \ref{generalalgorithm} terminates at $t^*=1$ because $\mathcal{O}_{t^*} = \mathcal{O}_0 \cap \mathrm{Pre}(\mathcal{O}_0)$ is contained entirely within $\mathrm{Pre}^2(\mathcal{O}_0)$, and is therefore positively invariant. All sets are shown projected onto the $x_1$-$x_2$ plane.}
		\label{fig:predemo}
	\end{figure}
	
	There is no guarantee that this algorithm will converge in a finite number of steps for a Markov chain with multiple eigenvalues of unit magnitude. A sufficient condition for finite determination of $\mathcal{O}_\infty(M,\mathcal{X})$ is that the Markov chain be \textit{ergodic} (i.e., irreducible and aperiodic), which is ensured by there existing some strictly positive probability vector $v$ such that $\lim_{k \rightarrow \infty}M^k = v \mathbf{1}^T$. This property not only ensures that $v$ is an eigenvector of $M$, but also that the error $e[k] = x[k]-v$ converges to zero with convergence rate $\rho(M-v\mathbf{1}^T)$.
	\begin{proposition} \label{proposition}
		Let $M$ be an ergodic Markov matrix. If there exists a probability vector $v \in \Delta$ such that $Mv = v$ and $Gv < g$, then Algorithm \ref{markovalgorithm} terminates after a finite number of iterations.
	\end{proposition}
	\begin{proof}
		It can be seen by induction that $M^k = (M-v\mathbf{1}^T)^k + v\mathbf{1}^T$ for all $k \geq 1$, so the condition $GM^kx \leq g \ \forall k \in \mathbb{N}$ is equivalent to $Gx\leq g$ and $G(M-v\mathbf{1}^T)^kx \leq g-Gv$ for all $k \geq 1$. If $Gv < g$, then there exists some scalar $\varepsilon > 0$ such that $\varepsilon \mathbf{1} \leq g - Gv$. Considering the induced $\infty$-norm $\vertiii{\cdot}_\infty$,
		\begin{subequations}
			\begin{equation}
			\vertiii{ G(M-v\mathbf{1}^T)^k x}_\infty \leq \vertiii{G}_\infty \vertiii{(M-v\mathbf{1}^T)^k}_\infty \vertiii{x}_\infty
			\end{equation}
			\begin{equation}
			\leq \vertiii{G}_\infty \vertiii{(M-v\mathbf{1}^T)^k}_\infty.
			\end{equation}
		\end{subequations}
		The last inequality follows from the fact that $\vertiii{x}_{\infty} \leq 1$ for all $x \in \Delta$. $\rho(M-v\mathbf{1}^T) < 1$ since $M$ is ergodic, implying that $\lim_{k\rightarrow\infty}(M-v\mathbf{1}^T)^k = 0$, so there must exist a $K$ such that $\vertiii{(M-v\mathbf{1}^T)^k}_\infty \leq \varepsilon/\vertiii{G}_\infty$ for all $k \geq K$. For $k \geq K$, $\vertiii{ G(M-v\mathbf{1}^T)^k x}_\infty \leq \varepsilon$, implying that $G(M-v\mathbf{1}^T)^k x \leq \varepsilon \mathbf{1} \leq g - Gv$, and so $GM^kx \leq g$ for all $x \in \Delta$.
		
		The stopping criterion for Algorithm \ref{markovalgorithm} is equivalent to that of Algorithm \ref{generalalgorithm}, both terminating when $\mathcal{O}_k \subseteq \mathrm{Pre}^{k+1}(\mathcal{X})$. This is ensured when $k \geq K-1$ because $\mathcal{O}_K \subseteq \Delta \subseteq \mathrm{Pre}^{K+1}(\mathcal{X})$.
	\end{proof}
	This constant $K$ in the proof of Proposition \ref{proposition} is a conservative estimate for the number of times the condition $\nexists Y \geq 0 \ s.t \ Y(G_k-g_k\mathbf{1}^T)\geq GM^{k+1}-g\mathbf{1}^T$ must be checked in Algorithm \ref{markovalgorithm}. The effect of $\rho(M-v\mathbf{1}^T)$ and $\varepsilon$ on $K$ can be found by examining the asymptotic behavior of $\Vert(M-v\mathbf{1}^T)^k\Vert_\infty$ as $k \rightarrow \infty$. By definition, $\rho(M-v\mathbf{1}^T) = \lim_{k \rightarrow \infty} \vertiii{(M-v\mathbf{1}^T)^k}_{\infty}^{1/k}$, and so $\vertiii{(M-v\mathbf{1}^T)^k}_{\infty}$ can be roughly approximated as $\rho(M-v\mathbf{1}^T)^k$. An estimate for $K$ is obtained by solving $\rho(M-v\mathbf{1}^T)^K = \varepsilon/\vertiii{G}_\infty$ for $K$, which results in $K \approx \frac{\log(\varepsilon / \vertiii{G}_\infty)}{\log(\rho)}$. The number of iterations should be lowest when $\varepsilon/\vertiii{G}_\infty$ is large and $\rho(M-v\mathbf{1}^T)$ is small, indicating that all initial conditions should rapidly converge to a safe neighborhood about $v$. However, this particular approximation for $K$ tends to be very conservative in practice, often leading to much smaller $k^*$ than would be predicted. A more refined approach may produce a better estimate of $k^*$, and give greater insight into how well this algorithm can be expected to perform.
	
	
	
	\section{Constrained Markov Chain Synthesis} \label{MarkovSynthesis}
	
	The previous section gave a way to compute the maximal invariant subset of a polyhedron $\Delta \cap \mathcal{P}(G,g)$ under the LTI dynamics $x^+ = Mx$, which was finitely determined under the assumptions of Proposition \ref{proposition}. We now focus on designing a Markov chain that satisfies these conditions.
	A Markov chain is considered to act on a directed graph with adjacency matrix $A_a$ defined such that $[A_a]_{i,j}=1$ if there is an edge from $S_i$ to $S_j$, and $[A_a]_{i,j}=0$ otherwise. To prevent unrealizable state transitions, the linear constraint
	\begin{equation} \label{eq:adjacency}
	M\odot\left(\mathbf{1}\mathbf{1}^T - A_a^T\right) = 0
	\end{equation}
	can be enforced along with other linear relations. This equation ensures that $M_{j,i}=0$ if there is no edge from $S_i$ to $S_j$, and there is no additional constraint on $M_{j,i}$ otherwise. $M$ alternatively may be defined as a sparse matrix with a sparsity pattern that is consistent with the underlying graph.
	The conditions for positive invariance in Theorem \ref{invariancetheorem} are linear in $Y$ and $M$ for a given pair $(G,g)$. However, this set may be too restrictive to generate a feasible $M$ for which $\Delta \cap \mathcal{P}(G,g)$ is positively invariant. Clearly if $Mv=v$ for some $v$ in $\mathcal{P}(G,g)$, then there exists an $x[0]$ such that $GM^kx[0] \leq g$ for all $k \in \mathbb{N}$, namely $x[0]=v$.

	As demonstrated by de Oliveira et al. \cite{de1999new}, all eigenvalues of $M-v\mathbf{1}^T$ have magnitude less than or equal to $\lambda \in \mathbb{R}_+$ iff there exist real matrices $P \succ 0$ and $D$ satisfying the bilinear matrix inequality
	\begin{equation} \label{BMI}
	\left[\begin{array}{cc}
	\lambda^2 P & (M-v\mathbf{1}^T)^TD^T \\ D(M-v\mathbf{1}^T) & D+D^T-P
	\end{array}\right] \succeq 0.
	\end{equation}
	If $D$ is held constant, condition \eqref{BMI} becomes linear in $P$. Then the spectral radius $\lambda$ can then be minimized with a line search over the interval $[0,1]$ to eliminate bilinearity, allowing the spectral radius to be minimized over Markov chains by solving a sequence of linear matrix inequalities (LMIs).
	However in the special case that $M$ is reversible, i.e., $M \mathrm{diag}(v) = \mathrm{diag}(v) M^T$, this condition may be replaced by the LMI
	\begin{equation} \label{eq:reversibleLMI}
	-\lambda I \preceq Q^{-1}MQ - rr^T \preceq \lambda I,
	\end{equation}
	with $r = v^{1/2}$ elementwise and $Q = \mathrm{diag}(r)$ \cite{acikmese2015TAC}. Although this condition is necessary and sufficient for reversible Markov matrices, it is not necessary in the general case, and may fail to find a feasible solution if one exists.

	\subsection{Synthesis Procedure}
	
	Let the following information be given: a polytope of constraints $\mathcal{P}(G,g)$, and a graph with adjacency matrix $A_a$ which satisfies the following assumption.
	\begin{assumption} \label{assumption1}
		There is some positive integer $l$ for which there exists a path of length $l$ between any two nodes.
	\end{assumption}
	Assumption \ref{assumption1} is necessary for a primitive $M$ to exist, ensuring ergodicity. If the steady-state distribution is specified as some particular $v > 0$, then the constraint $Mv=v$ is linear.
	If $v$ is not specified, then a second assumption is made.
	\begin{assumption} \label{assumption2}
		There exists some $v\in\Delta$ such that $v>0$ and $Gv < g$.
	\end{assumption}
	If such a $v$ exists, then $D$ can be set to $\mathrm{diag}(v)^{-1}$ as before. These two assumptions ensure that there exists some Markov matrix $M$ for which $\lim_{k \rightarrow \infty}M^k = v\mathbf{1}^T$, and that Algorithm \ref{markovalgorithm} terminates as a result of Proposition \ref{proposition} from the existence of some $\varepsilon >0$ such that $g-Gv \geq \varepsilon \mathbf{1}$. Furthermore if $A_a$ is symmetric, then an ergodic, reversible Markov chain can be explicitly constructed, e.g., with the Metropolis-Hastings algorithm, which would ensure that a feasible solution exists \cite{acikmese2015markov}.

	If the steady-state distribution is not given, then the first step is to choose a $v \in \Delta$ that satisfies the strict inequalities $v > 0$ and $Gv < g$. Then find the smallest $\lambda \geq 0$ for which there exists $M \in \mathbb{R}^{n \times n}$ subject to:
	\begin{subequations} \label{convexset}
		\begin{equation}
		M \geq 0,
		\end{equation}
		\begin{equation}
		\mathbf{1}^T M = \mathbf{1}^T,
		\end{equation}
		\begin{equation}
		Mv = v,
		\end{equation}
		\begin{equation}
		M \odot (\mathbf{1}\mathbf{1}^T - A_a^T)  = 0,
		\end{equation}
		\begin{equation} \label{eq:ReversibleCondition}
		-\lambda I \preceq Q^{-1}MQ - rr^T \preceq \lambda I.
		\end{equation}
	\end{subequations}
	If no reversible ergodic Markov matrix is found, then this problem can be modified by replacing \eqref{eq:ReversibleCondition} with LMI \eqref{BMI} and $P \geq 0$. Bilinearity can be avoided by fixing $D$, e.g., $D = \mathrm{diag}(v)^{-1}$ as suggested in \cite{acikmese2015markov}, and testing various values of $\lambda$ in the interval $[0,1]$ to find the minimum.
	
	This procedure requires solving a sequence of linear feasibility problems in the matrix variable $M$. Convex programming solvers can minimize a convex function of $M$ subject to constraints \eqref{convexset}, and additional convex constraints can be imposed such as upper or lower bounds on each element of $M$ to directly specify a range of desirable transition probabilities. One possibility is to minimize the frequency of state transitions, $\sum_{i=1}^n (1-M_{i,i})v_i$, to discourage the system from changing states too often. This quantity is the probability of a transition occurring while the system has converged to $v$.
	
	Directly solving for an optimal $v$ may not scale well with dimension, as it appears bilinearly in the constraints. There may be a tradeoff in performance between not letting any element of $v$ from being too close to $0$, and maximizing the smallest entry of $g - Gv$. The maximal invariant set computation may not terminate if $M$ has multiple eigenvalues of magnitude $1$. We minimize $\rho(M - v\mathbf{1}^T)$ as a tractable heuristic to compute $\mathcal{O}_\infty$ in as few steps as possible, however further research may better estimate the number of steps the algorithm will terminate given particular $M$ and $\mathcal{P}(G,g)$.

	\section{Example}
	
	This example is a decentralized swarm guidance problem, where there are many agents which each have the same stochastic control scheme based on a state-dependent probability distribution. Each agent must choose which control action to perform based only on knowledge of its current location, and limited collision avoidance capabilities. The sequence of bins each agent visits follows a Markov chain, and if all agents follow the same randomized policy, $x[k]$ can be interpreted as the expected distribution of agents at step $k$. The Markov chain must be designed to converge safely to a desired steady-state distribution.
	
	The bins are arranged in a grid, with some bins containing obstacles which must be avoided. No cell may have more than 30\% probability mass at any time. The steady-state distribution is such that each agent is expected to spent 90\% of its time in one of the four terminal states, regardless of their initial location. The steady-state distribution $v$ is $22.5\%$ at each of the four terminal bins and $0.25\%$ elsewhere.
	The convergence rate was optimized by checking the feasibility of the set defined in \eqref{convexset} using SeDuMi, with $\lambda^* = 0.9950$ being optimal. With this value of $\lambda$, the particular $M$ was found by maximizing the linear objective function $\sum_{i}M_{i,i}v_{i}$, the rate at which agents remain in their same bin. Algorithm \ref{markovalgorithm} was run to determine set of safe initial conditions, terminating with $t^* = 3$.
	We verify that the distribution in Figure \ref{fig:safetyverification} is in $\mathcal{O}_\infty$ by checking that $x[k]\leq (0.3)\mathbf{1}, \ \forall k\in\{0,1,2,3\}$. This distribution is propagated forward to $k=100$ and $k=1000$, as shown in Figure \ref{fig:t100t1000}, which confirm that the safety constraints are always satisfied.
	

	\begin{figure}
		\centering
		\includegraphics[width=0.8\linewidth]{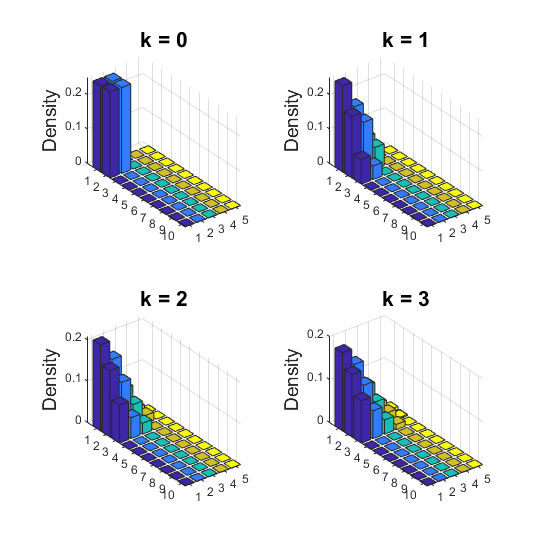}
		\caption{The initial distribution $x[0]$ is shown on the top left. Safety is verified for all subsequent time steps by observing that the safety constraints are satisfied for all $k=0,\dots,t^*$.}
		\label{fig:safetyverification}
	\end{figure}

	\begin{figure}
		\centering
		\subfloat[k = 100]{{\includegraphics[width=0.49\linewidth]{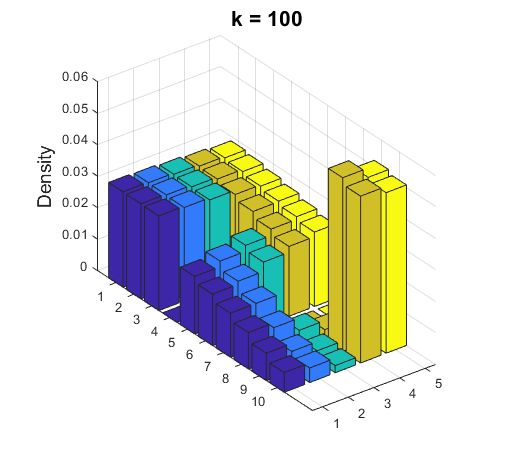} }}%
		\qquad
		\subfloat[k = 1000]{{\includegraphics[width=0.4\linewidth]{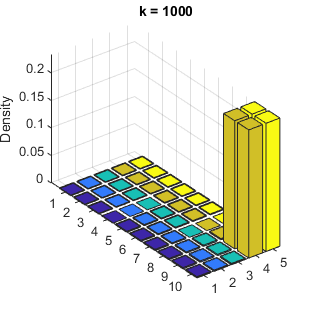} }}%
		\caption{Transient and steady-state behavior with the initial distribution shown in Figure \ref{fig:safetyverification}.}%
		\label{fig:t100t1000}%
	\end{figure}
	
%


	\section{Conclusion}
	
	We have specialized a general maximal invariant set computation algorithm for Markov chains subject to polyhedral safety constraints on the probability distribution over states, and we have given conditions which ensure that the maximal invariant set is finitely determined and polyhedral. We then gave an SDP-based procedure for synthesizing a Markov chain which guarantees finite determination by proper choice of the steady-state distribution and by minimizing the second largest eigenvalue of the Markov matrix $M$, by promoting rapid convergence to a safe steady-state distribution. This method was illustrated with a swarm exploration example.

\bibliographystyle{unsrt}
\bibliography{InvariantSet}

\end{document}